\documentclass[a4paper,12pt]{article}

\usepackage[margin=1.2in]{geometry}

\usepackage{amsmath}
\usepackage{algorithmic}
\usepackage{algorithm}
\usepackage{amsthm}
\usepackage{amsfonts}
\usepackage{comment}
\usepackage{graphicx}
\usepackage{hyperref}
\usepackage{mathrsfs}
\usepackage{color}

\newtheorem{theorem} {Theorem}
\newtheorem{lemma} {Lemma}

\newtheorem{remark} {Remark}
\newtheorem{conjecture} {Conjecture}

\def\sA{{\mathscr{A}}}

\def\m{{\mathbf{m}}}
\def\x{{\mathbf{x}}}
\def\c{{\mathbf{c}}}
\def\g{{\mathbf{g}}}
\def\e{{\mathbf{e}}}

\def\v{{\mathbf{v}}}
\def\z{{\mathbf{z}}}
\def\w{{\mathbf{w}}}
\def\y{{\mathbf{y}}}

\def\b{{\mathbf{b}}}
\def\X{{\mathbf{X}}}

\def\A{{\mathbf{A}}}

\def\I{{\mathbf{I}}}

\def\Z{{\mathbf{Z}}}

\def\G{{\mathbf{G}}}

\newcommand{\mA}{\mathcal{A}}

\newcommand{\mC}{\mathcal{C}}

\newcommand{\mB}{\mathcal{B}}
\newcommand{\mbS}{\mathbb{S}}

\newcommand{\E}{\mathbb{E}}

\newcommand{\trace}{\textrm{Tr}}

\newcommand{\reals}{\mathbb{R}}

\title{From Oja's Algorithm to the Multiplicative Weights Update Method with Applications}
\date{}
\author{Dan Garber \\ Technion - Israel Institute of Technology \\ \small{dangar@technion.ac.il}}

\begin{document}

 \maketitle

\begin{abstract}
Oja's algorithm is a well known online algorithm studied mainly in the context of stochastic principal component analysis. We make a simple observation, yet to the best of our knowledge a novel one, that when applied  to a any (not necessarily stochastic) sequence of  symmetric matrices which share common eigenvectors, the regret of Oja's algorithm could be directly bounded in terms of the regret of the well known multiplicative weights update method for the problem of prediction with expert advice. Several applications to optimization with quadratic forms over the unit sphere in $\reals^n$ are discussed. 
\end{abstract}

\section{Introduction}
Fast approximation of a leading eigenvector of a large real symmetric matrix is one of the most important computational primitives which underlies several important applications, such as principal component analysis and the pagerank algorithm, and is also a fundamental sub-procedure in many optimization algorithms, in particular in the context of semidefinite optimization. Fast iterative approximation methods include the well known power  method and the often faster yet also more involved Lanczos-type algorithms, as well as other related methods \cite{golub2013matrix, saad2011numerical}.
Such methods, which are very efficient for solving a very specific large scale non-convex optimization problem --- globally  maximizing a quadratic form over the unit sphere in $\reals^n$ without any further assumptions on the input matrix, admit a highly specialized analysis which often relies only on elementary arguments in linear algebra, such as in the case of the power method \cite{golub2013matrix, saad2011numerical}. While some of these methods could in principle also be seen as first-order methods of continuous optimization for solving the associated non-convex optimization problem, their convergence analysis, which guarantees rapid convergence to the global optimal solution (with high probability), often does not easily fit with the theory of first-order methods for non-convex problems, which mostly focuses on convergence only to certain stationarity criteria and do not yield globally optimal solutions \cite{beck2017first, pmlr-v48-shamira16, beck2018globally}. Such methods have also been studied from a Riemannian optimization point of view which leads to similar convergence rates to that of the power method \cite{xu2018convergence, ding2021tight}, however their analysis is substantially more involved.

In recent years, the stochastic variant of the leading eigenvector problem, in which the input matrix $\A$ is not observed directly, but only a finite sequence of i.i.d. samples $\A_1,\dots,\A_T$ such that $\E[\A_t]=\A$ for all $t\in[T]$ is observed, has received significant interest. Oja's algorithm \cite{oja1982simplified, oja1985stochastic}, see Algorithm \ref{alg:oja} below, which can be seen as an instance of stochastic gradient ascent (w.r.t. the maximization of the objective function $f(\z) = \z^{\top}\A\z$ over the unit sphere), was proposed for this setting and was studied extensively in recent years \cite{balsubramani2013fast, pmlr-v48-shamirb16, allen2017follow}.

Let us now consider a seemingly very different problem than computing leading eigenvectors. In the problem of prediction with expert advice \cite{littlestone1994weighted, freund1997decision}, a decision maker (DM) is iteratively required to choose an expert among $n$ possibilities. At the end of each round, each expert incurs a scalar loss and the DM incurs the loss of the expert she chose on that round. The goal of the DM is to achieve cumulative loss (throughout all rounds,, whose number is finite and assumed to be known in advanced) not much larger than that of the best expert in hindsight. The difference between the overall loss of the DM and that of the best expert is  called regret, and the goal is to guarantee that it grows only sub-linearly with the number of rounds $T$ for any sequence of bounded losses (e.g., each loss is in $[-1,1]$). The multiplicative weights update method (MW) \cite{arora2012multiplicative} is one of the most well known algorithms for this problem. MW maintains a positive weight for each expert, which is updated each round according to the loss of the expert, and prediction on each round is made by randomly selecting an expert according to the distribution induced by the weights, see Algorithm \ref{alg:mw} below. For losses bounded in $[-1,1]$, MW guarantees expected regret of the order $\sqrt{T\log{n}}$, and the beautiful proof of this result is both very short and highly elementary, see for instance Theorem 2.1 in  \cite{arora2012multiplicative}.

In this short communication we consider Oja's algorithm as an online optimization algorithm, that is, the input sequence of matrices $\A_1,\dots,\A_T$ should not be thought of as i.i.d. samples, but as some arbitrary sequence in the space of real symmetric matrices $\mbS^n$. We make an observation that when these inputs  share common eigenvectors, by considering a simple transformation, the dynamics of Oja's algorithm could be shown to correspond exactly to the MW algorithm with a particular sequence of loss vectors. This observation immediately yields a regret bound for Oja's algorithm by leveraging the well known regret bound of MW. As a first straightforward application, we obtain provable convergence of Oja's algorithm for (deterministic) leading eigenvector approximation that follows directly from this regret bound, as opposed to typical power method-style analysis. This establishes a somewhat surprising connection between the two fundamental yet very distinct problems of eigenvector computation and prediction with expert advice, and offers a fresh view on fundamental eigenvalue computations which follows from elementary arguments. 
As a more advanced application we show how this result in fact applies to more general optimization problems involving convex minimization with certain quadratic forms over the unit sphere, yielding a globally convergent (with high probability) and highly efficient first-order method for these nonconvex problems. Finally, we also make a connection back to online learning, and discuss implications to the possibility of designing  highly efficient algorithms to an online learning version of the leading eigenvector problem studied in recent years \cite{pmlr-v37-garberb15, nie2016online, allen2017follow, pmlr-v99-carmon19a, garber2019regret}.

%\textbf{Equivalences between MD and GD:} \cite{gunasekar2017implicit, amid2020winnowing, amid2020reparameterizing, ghai2022non}
%\\

%\noindent\textbf{Matrix Multiplicative Weights:}

\begin{minipage}{0.47\textwidth}
\begin{algorithm}[H]
	\caption{Oja's Algorithm}\label{alg:oja}
	\begin{algorithmic}
		\STATE \textbf{Input:} step-size $\mu\ge0$ 
		\STATE \textbf{Init:} $\w_1 \gets$ rand. unit vector in $\reals^n$
		\FOR{$t = 1,2,...$}
		\STATE observe gain matrix $\A_t\in\mbS^n$
		\STATE $\z_{t}\gets \w_t / \Vert{\w_t}\Vert_2$ 
            \STATE $\w_{t+1} \gets (\I+\mu\A_t)\w_t$   
			
        \ENDFOR
	\end{algorithmic}
\end{algorithm}
\end{minipage}
\hfill
\begin{minipage}{0.47\textwidth}
\begin{algorithm}[H]
	\caption{Multiplicative Weights}\label{alg:mw}
	\begin{algorithmic}
		\STATE \textbf{Input:} step-size $\eta\ge0$ 
		\STATE \textbf{Init:} $\phi_1\in\reals_{>0}^n$
		\FOR{$t = 1,2,...$} 
		\STATE observe loss vector $\m_t\in\reals^n$
		%\FOR{$i=1,\dots,n$}
           % \STATE $\p_t \gets \phi_t/\Vert{\phi_t}\Vert_1$   
			\STATE $\phi_{t+1}(i) \gets \phi_t(i)(1-\eta{}\m_t(i))$  for all $i=1,\dots,n$
			%\ENDFOR
        \ENDFOR
	\end{algorithmic}
\end{algorithm}
\end{minipage}

\subsection{Notation}
We let $\reals_{>0}^n$ denote the restriction of $\reals^n$ to vectors with positive entries. We use lower-case boldface letters to denote vectors in $\reals^n$ and upper-case boldface letters to denote matrices. Scalars are denoted by lightface letters. For real matrices we let $\Vert{\cdot}\Vert_2$ denote the spectral norm (largest singular value) and for vectors in $\reals^n$ we let $\Vert{\cdot}\Vert_2, \Vert{\cdot}\Vert_1, \Vert{\cdot}\Vert_{\infty}$ denote the Euclidean norm, the $\ell_1$ norm, and the $\ell_{\infty}$ norm, respectively.

\section{Regret bound for Oja's Algorithm under Common Eigenvectors Assumption}
We begin our study with stating the well known regret bound for MW (Algorithm \ref{alg:mw}).
The following theorem is taken from \cite{arora2012multiplicative} with the obvious change of notation and with the minor change that while \cite{arora2012multiplicative} assumes $\phi_1$ is the all-ones vector, we let it be an arbitrary vector of positive entries. The modification to the (short and well known) proof is straightforward and thus we omit it. 
\begin{theorem}[Theorem 2.1 in \cite{arora2012multiplicative}]\label{thm:mw}
Consider Algorithm \ref{alg:mw} with learning rate $\eta\in[0,1/2]$ and with $\m_t(i)\in[-1,1]$ for all $i\in[n]$, $t\in[T]$. Then, for all $i\in[n]$ the following regret bound holds:
\begin{align*}
\sum_{t=1}^T\left({\left({\frac{\phi_t}{\Vert{\phi_t}\Vert_1}}\right)^{\top}\m_t -\m_t(i)}\right) \leq \eta\sum_{t=1}^T\vert{\m_t(i)}\vert + \frac{\ln\left({\Vert{\phi_1}\Vert_1\phi_1(i)^{-1}}\right)}{\eta}.
\end{align*}
\end{theorem}

Consider now a finite sequence of symmetric matrices $\{\A_t\}_{t=1}^T\subset\mbS^n$ which all share the same full set of orthonormal eigenvectors $\v_1,\dots,\v_n$, and suppose further that $\Vert{\A_t}\Vert_2 \leq 1$ for all $t\in[T]$. That is, for every $t\in[T]$, we can write the eigen-decomposition of $\A_t$ as $\A_t = \sum_{i=1}^n\lambda_{t}(i)\v_i\v_i^{\top}$, where $\lambda_t\in\reals^n$ with $\Vert{\lambda_{t}}\Vert_{\infty} \leq 1$.

Denote for all $t\in[T]$ and $i\in[n]$, $\psi_t(i) = (\v_i^{\top}\w_t)^2$. Using the update step of Algorithm \ref{alg:oja} we have for all $t\in[T]$ and $i\in[n]$,
\begin{align}\label{eq:oja2mw:1}
\psi_{t+1}(i) & = \left({\v_i^{\top}(\I+\mu\A_t)\w_t}\right)^2 =(\v_i^{\top}\w_t )^2(1+\mu\lambda_{t}(i))^2 \nonumber \\
&= \psi_t(i)\left({1+\mu\left({2\lambda_{t}(i)+\mu\lambda_{t}(i)^2}\right)}\right).
\end{align}
Suppose $\mu = \frac{\eta}{C}$ for some $\eta\in[0,1/2]$ and $C>0$ such that for all $t\in[T]$ and $i\in[n]$,
\begin{align*}
\frac{2\lambda_{t}(i)+\frac{\eta}{C}\lambda_{t}(i)^2}{C}\in[-1,1].
\end{align*}

In particular, under the assumption above that $\Vert{\lambda_t}\Vert_{\infty} \leq 1$,  we can  conveniently set $C=3$, meaning $\mu\in[0,1/6]$.

Now, from Eq. \eqref{eq:oja2mw:1} we clearly see that the sequence $(\psi_t)_{t\geq 1}$ follows the same dynamics as the sequence $(\phi_t)_{t\geq 1}$ in Algorithm \ref{alg:mw} with step-size $\eta$ and losses:
\begin{align*}
\m_t(i) = - \frac{2\lambda_{t}(i)+\frac{\eta}{3}\lambda_{t}(i)^2}{3} \quad \forall t\in[T],~i\in[n].
\end{align*}
Applying Theorem \ref{thm:mw} and slightly rearranging we immediately obtain that for all $i\in[n]$,
\begin{align*}
-\sum_{t=1}^T\left({\frac{\psi_t}{\Vert{\psi_t}\Vert_1}}\right)^{\top}\lambda_t + \sum_{t=1}^T\lambda_{t}(i) &\leq \frac{\eta}{2}\sum_{t=1}^T\left|{2\lambda_{t}(i)+\frac{\eta}{3}\lambda_{t}(i)^2}\right| \\
&~+ \frac{\eta}{6}\sum_{t=1}^T\left({\left({\frac{\psi_t}{\Vert{\psi_t}\Vert_1}}\right)^{\top}\lambda_t^2 - \lambda_{t}(i)^2}\right) \\
&~+ \frac{3\ln\left({\Vert{\psi_1}\Vert_1\psi_1(i)^{-1}}\right)}{2\eta},
\end{align*}
where $\lambda_t^2\in\reals^n$ is the vector whose entries are the squares of $\lambda_t$.

Recalling that $\eta = 3\mu$, $\eta\in[0,1/2]$, and  $\Vert{\lambda_t}\Vert_{\infty}\leq 1$ for all $t\in[T]$, the above bound could be simplified as follows:
\begin{align}\label{eq:oja2mw:2}
-\sum_{t=1}^T\left({\frac{\psi_t}{\Vert{\psi_t}\Vert_1}}\right)^{\top}\lambda_t + \sum_{t=1}^T\lambda_{t}(i) &\leq 3\mu\sum_{t=1}^T\left|{\lambda_{t}(i)}\right| + \frac{\mu}{2}\sum_{t=1}^T\Vert{\lambda_t}\Vert_{\infty}^2 \nonumber \\
&~+ \frac{\ln\left({\Vert{\psi_1}\Vert_1\psi_1(i)^{-1}}\right)}{2\mu}.
\end{align}

Recall that $\psi_t(i) = (\v_i^{\top}\w_t)^2$. Thus, $\Vert{\psi_t}\Vert_1 = \Vert{\w_t}\Vert_2^2$ and in particular $\Vert{\psi}\Vert_1 = 1$. Moreover,
\begin{align*}
\left({\frac{\psi_t}{\Vert{\psi_t}\Vert_1}}\right)^{\top}\lambda_t = \sum_{i=1}^n\frac{(\v_i^{\top}\w_t)^2\lambda_{t}(i)}{\Vert{\w_t}\Vert_2^2}
= \frac{1}{\Vert{\w_t}\Vert_2^2}\w_t^{\top}\left({\sum_{i=1}^n\lambda_{t}(i)\v_i\v_i^{\top}}\right)\w_t = \z_t^{\top}\A_t\z_t.
\end{align*}

Plugging these observations into Eq. \eqref{eq:oja2mw:2} and recalling that $\lambda_t(i) = \v_i^{\top}\A_t\v_i$ for all $t\in[T], i\in[n]$, we immediately obtain the following result.
\begin{theorem}\label{thm:oja}
Consider Algorithm \ref{alg:oja} in case the matrices $\A_1,\dots,\A_T\in\mbS^n$ share a common full set of orthonormal eigenvectors $\v_1,\dots,\v_n$, and have eigenvalues bounded in $[-1,1]$. Then, for any step-size $\mu\in[0,1/6]$ and any $i\in[n]$, we have
\begin{align*}
\sum_{t=1}^T\left({\v_i^{\top}\A_t\v_i-\z_t^{\top}\A_t\z_t}\right)  &\leq \mu\sum_{t=1}^T\left({3\left|{\v_i^{\top}\A_t\v_i}\right| + \frac{1}{2}\Vert{\A_t}\Vert_{2}^2}\right) - \frac{\ln\left({(\z_1^{\top}\v_i)^2}\right)}{2\mu}.
%\sum_{t=1}^T\lambda_{t}(i)-\sum_{t=1}^T\z_t^{\top}\A_t\z_t  &\leq 3\mu\sum_{t=1}^T\left|{\lambda_{t}(i)}\right| + \frac{\mu}{2}\sum_{t=1}^T\Vert{\lambda_t}\Vert_{\infty}^2 - \frac{\ln\left({(\z_1^{\top}\v_i)^2}\right)}{2\mu}.
\end{align*}
\end{theorem}

\subsection{A slightly improved MW-inspired analysis}
The well known proof for the MW algorithm (see \cite{arora2012multiplicative}) is based on the potential  $\Phi_{T+1} := \sum_{i=1}^n\phi_{T+1}(i)$ (where for all $t\in[T]$, $\phi_{t}$ is as defined in Algorithm \ref{alg:mw}) which is used to derive lower and upper bounds on the cumulative loss of the algorithm and that of each expert. In  case of the analysis above we have $\Phi_{T+1} = \sum_{i=1}^n\psi_{T+1}(i) = \sum_{i=1}^n(\v_i^{\top}\w_{T+1})^2 = \Vert{\w_{T+1}}\Vert_2^2$. This suggests analyzing the regret of Oja's algorithm directly in terms of the potential $\Vert{\w_{T+1}}\Vert_2^2$. Indeed, mimicking the simple idea in the proof of Theorem \ref{thm:mw} (deriving upper and lower bounds on $\Vert{\w_{T+1}}\Vert_2^2$), we can write
\begin{align*}
\Vert{\w_{T+1}}\Vert_2^2 &= \Vert{\w_T + \eta\A_T\w_T}\Vert_2^2  = \Vert{\w_T}\Vert_2^2\Vert{\z_T + \mu\A_T\z_T}\Vert_2^2\\
 &= \Vert{\w_T}\Vert_2^2\left({1 + 2\mu\z_T^{\top}\A_T\z_T + \mu^2\z_T\A_T^2\z_T}\right) \\
 &\leq \Vert{\w_T}\Vert_2^2{}\exp\left({2\mu\z_T^{\top}\A_T\z_T + \mu^2\z_T\A_T^2\z_T}\right).
\end{align*}
Unrolling the recursion, recalling that $\Vert{\w_1}\Vert_2 =1$, and taking the natural logarithm on both sides, we obtain
\begin{align}\label{eq:newAnal:1}
\log\left({\Vert{\w_{T+1}}\Vert_2^2}\right) \leq \mu\sum_{t=1}^T2\z_t^{\top}\A_t\z_t + \mu\z_t^{\top}\A_t^2\z_t.
\end{align}

Let $\v\in\reals^n$ be a unit vector, and assume  it is a common eigenvector of all  matrices $\A_1,\dots,\A_T$. Then,
\begin{align*}
(\v^{\top}\w_{T+1})^2 &= \left({\v^{\top}\left({\w_T + \mu\A_T\w_T}\right)}\right)^2 = \left({\v^{\top}\w_T + \mu\v^{\top}\A_t\v\v^{\top}\w_T}\right)^2 \\
&= (\v^{\top}\w_T)^2(1+\mu\v^{\top}\A_T\v)^2.
\end{align*}

Unrolling the recursion  and taking the natural logarithm on both sides, we obtain
\begin{align}\label{eq:newAnal:2}
\log\left({\Vert{\w_{T+1}}\Vert_2^2}\right) & \geq \log\left({(\v^{\top}\w_{T+1})^2}\right) \nonumber \\
&= \log\left({(\v^{\top}\w_1)^2}\right) + 2\sum_{t=1}^T\log\left({1+\mu\v^{\top}\A_t\v}\right) \nonumber \\
&\geq \log\left({(\v^{\top}\z_1)^2}\right) + 2\mu\sum_{t=1}^T\left({\v^{\top}\A_t\v - \mu\left({\v^{\top}\A_t\v}\right)^2}\right),
\end{align}
where in the last inequality we have used the assumption that $\mu\in[0,1/2]$, $\Vert{\A_t}\Vert_2 \leq 1$, and the inequality $\log(1+x) \geq x-x^2$ for $x\in[0,1/2]$.

Combining Eq. \eqref{eq:newAnal:1} and \eqref{eq:newAnal:2} we obtain,
\begin{align*}
\sum_{t=1}^T\z_t^{\top}\A_t\z_t &\geq  \sum_{t=1}^T\v^{\top}\A_t\v  - \frac{\mu}{2}\sum_{t=1}^T\z_t^{\top}\A_t^2\z_t  - \mu\sum_{t=1}^T(\v^{\top}\A_t\v)^2 + \frac{1}{2\mu}\log\left({(\v^{\top}\z_1)^2}\right) \\
&\geq  \sum_{t=1}^T\v^{\top}\A_t\v - \frac{3\mu}{2}\sum_{t=1}^T\Vert{\A_t}\Vert_2^2 + \frac{1}{2\mu}\log\left({(\v^{\top}\z_1)^2}\right). 
\end{align*}
This leads to the following theorem which slightly improves upon Theorem \ref{thm:oja} by not only providing slightly better constants, but also in slightly relaxing the assumption that the inputs $\A_1,\dots,\A_T$ share a common full set of orthonormal eigenvectors, to only assuming that the comparator vector $\v$ is a mutual eigenvector.
\begin{theorem}\label{thm:oja2}
Consider running Algorithm \ref{alg:oja} w.r.t. input matrices $\A_1,\dots,\A_T$ in $\mbS^n$ that have eigenvalues in $[-1,1]$ and share a common unit eigenvector $\v$. Then, for any step-size $\mu\in[0,1/2]$ it holds that
\begin{align*}
\sum_{t=1}^T\left({\v^{\top}\A_t\v-\z_t^{\top}\A_t\z_t}\right) &\leq  \frac{3\mu}{2}\sum_{t=1}^T\Vert{\A_t}\Vert_2^2 - \frac{1}{2\mu}\log\left({(\v^{\top}\z_1)^2}\right).
%
%\sum_{t=1}^T\left({\v^{\top}\A_t\v-\z_t^{\top}\A_t\z_t}\right) &\leq  \mu\sum_{t=1}^T\left({\v^{\top}\A_t^2\v + \frac{1}{2}\Vert{\A_t}\Vert_2^2}\right) - \frac{1}{2\mu}\log\left({(\v^{\top}\z_1)^2}\right).
\end{align*}
\end{theorem}

%\subsubsection{New attempt at block version}
%\begin{align*}
%\Vert{\W_{t+1}}\Vert_F^2 &= \Vert{\W_t + \eta\A_t\W_t}\Vert_F^2 =  \Vert{\W_t}\Vert_F^2 + \trace(\bR_t^{\top}\Q_t^{\top}(2\eta\A_t+\eta^2\A_t^2)\Q_t\bR_t) \\
%&\leq \Vert{\W_t}\Vert_F^2 + \Vert{\W_t}\Vert_F^2\trace(\Q_t^{\top}(2\eta\A_t+\eta^2\A_t^2)\Q_t).
%\end{align*}
%On the other-hand,

%\begin{align*}
%\Vert{\Q^{\top}\W_{t+1}}\Vert_F^2 = \Vert{\Q^{\top}\W_t + \eta\Lambda_t\Q^{\top}\W_t}\Vert_F^2
%= \Vert{\Q^{\top}\W_t}\Vert_F^2 + \trace(\W_t^{\top}\Q(2\eta\Lambda_t +\eta^2\Lambda_t^2)\Q^{\top}\W_t)
%\end{align*}

\section{Applications}
For all applications discussed we will need to lower-bound the term $(\z_1^{\top}\v)^2$ appearing in Theorem \ref{thm:oja2}. This is carried out using the following  standard argument.
\begin{lemma}[Lemma 5 in \cite{arora2009expander}]\label{lem:corr}
If $\z\in\reals^n$ is a random unit vector sampled from the uniform distribution over the unit sphere, then for any unit vector $\v$ and any $\delta\in(0,1)$, it holds with probability at least $1-\delta$ that $(\z^{\top}\v)^2 \geq \frac{\delta}{9n}$.
\end{lemma}

\subsection{Leading eigenvalue approximation}\label{sec:EVapprox}
Consider Algorithm \ref{alg:oja} in the case that for all $t\in[T]$, $\A_t = \A$, for some $\A\in\mbS^n$ with $\Vert{\A}\Vert_2\leq 1$, and denote by $\v_1$ a leading eigenvector of $\A$ associated with the largest (signed) eigenvalue $\lambda_1(\A)$ . Clearly, in this setting Oja's algorithm with step-size $\mu\in[0,1]$ corresponds exactly to the well-known \textit{power method} \cite{golub2013matrix, saad2011numerical} when applied to the positive semidefinite matrix $\I+\mu\A$ for which $\v_1$ is also a leading eigenvector. Theorems  \ref{thm:oja} and \ref{thm:oja2} provide an alternative  to the standard power method analysis, immediately yielding an approximation guarantee for $\lambda_1(\A)$. Consider running Algorithm \ref{alg:oja} for $T$ iterations with step-size $\mu\in[0,1/2]$ and picking the iterate $\z_t$ for which the product $\z_t^{\top}\A\z_t$ is maximized. Theorem \ref{thm:oja2} implies that
\begin{align*}
\max_{t\in[T]}\z_t^{\top}\A\z_t \geq \frac{1}{T}\sum_{t=1}^T\z_t^{\top}\A\z_t \geq \lambda_1(\A) - \frac{3\mu}{2} + \frac{\ln((\v_1^{\top}\z_1)^2)}{2\mu{}T}.
\end{align*}

Using Lemma \ref{lem:corr}, we have that for any $\delta\in(0,1)$, setting $\mu = \sqrt{\frac{\ln\left({9n/\delta}\right)}{3T}}$ and assuming $T$ is large enough (so that $\mu \leq 1/2$), it holds with probability at least $1-\delta$ that
\begin{align}\label{eq:mwEV}
\max_{t\in[T]}\z_t^{\top}\A\z_t \geq \lambda_1(\A) - \sqrt{\frac{3\ln\left({9n/\delta}\right)}{T}}.
\end{align}

In particular, to reach $\epsilon$ additive approximation with probability at least $1-\delta$, $T\geq \lceil{3\epsilon^{-2}\ln\left({9n/\delta}\right)}\rceil$ iterations are required.

\begin{remark}
The approximation error in \eqref{eq:mwEV} scales (logarithmically) with the dimension $n$. Interestingly, this worst-case dependence is well known for the MW algorithm (see for instance Theorem 2.1 in \cite{arora2012multiplicative}), but also comes up in standard analysis of the power method with random uniform initialization (see for instance Theorem 2.1 in \cite{garber2015fast}). In fact, it was shown in  \cite{simchowitz2018tight} that such dependence is unavoidable in worst case for leading eigenvector approximation in the matrix-vector product oracle model.  
\end{remark}

\begin{remark}
While the classical analysis of the power method is often simplified by the assumption that $\lambda_1(\A) > \lambda_2(\A)$, which leads to a typical $\frac{\lambda_1(\A)}{\lambda_1(\A)-\lambda_2(\A)}\log\frac{1}{\epsilon}$ iteration complexity (omitting for simplicity the dependence on the probability of success and assuming $\A$ is positive semidefinite) to find a $(1-\epsilon)$ approximation \cite{golub2013matrix, saad2011numerical}, here with our MW-based analysis of Oja's algorithm, we naturally obtain a result for the ``gap-free'' regime in which no assumption on the uniqueness of the leading eigenvalue is made, but the iteration complexity scales polynomially with $1/\epsilon$ and not logarithmically.
\end{remark}
\begin{remark}
While standard analysis of the power method  \cite{golub2013matrix, saad2011numerical} naturally leads to convergence rates w.r.t. the last iterate, our MW-based analysis naturally leads to regret-style convergence rates which hold w.r.t. the average of iterates  (which in turn imply rates w.r.t. to the best iterate in terms of the objective, as in Eq. \eqref{eq:mwEV}).
\end{remark}

%Consider the case that $\A\succeq 0$. In this case we have
%\begin{align*}
%\max_{t\in[T]}\z_t^{\top}\A\z_t \geq \lambda_1(\A)\left({1 - \frac{7\mu}{2} + \frac{\ln(\z_1^{\top}\v_1)^2}{2\mu{}\lambda_1(\A)T}}\right).
%\end{align*}

%In this case, given estimate $\hat{\lambda}_1$ for $\lambda_1(\A)$, setting $\mu = \sqrt{\frac{\ln\left({\frac{9n}{\delta}}\right)}{7\hat{\lambda}_1T}}$ and assuming $T$ is large enough (so that $\mu \leq 1/6$), we have that

%\begin{align*}
%\max_{t\in[T]}\z_t^{\top}\A\z_t \geq \lambda_1(\A)\left({1 -  \frac{1}{2}\sqrt{\frac{7\ln\left({\frac{9n}{\delta}}\right)}{\hat{\lambda}_1T}}- \frac{1}{2}\sqrt{\frac{7\hat{\lambda}_1\ln\left({\frac{9n}{\delta}}\right)}{\lambda_1(\A)^2T}}}\right)
%\end{align*}

%Thus if $\hat{\lambda}_1 = \Theta(\lambda_1(\A))$ we have that
%\begin{align*}
%\max_{t\in[T]}\z_t^{\top}\A\z_t \geq \lambda_1(\A)\left({1 - \Theta\left({\sqrt{\frac{\ln\left({\frac{n}{\delta}}\right)}{\lambda_1(\A){}T}}}\right)}\right).
%\end{align*}

%Now, the number of iterations to reach $1-\epsilon$ multiplicative approximation scales with  $\frac{1}{\lambda_1(\A)\epsilon^2}$ as opposed to $\frac{1}{\lambda_1(\A)^2\epsilon^2}$ in the previous case.

\subsection{Convex minimization with quadratic mapping over the unit sphere}
Let $\mA:\reals^n\rightarrow\reals^m$ be a quadratic map of the form: $(\mA(\x))_i = \x^{\top}\A_i\x$, $i=1,\dots,m$, where  throughout the discussion we assume that $\A_1,\dots,\A_m\in\mbS^n$ and share a common set of orthonormal eigenvectors. We further assume that $\Vert{\A_i}\Vert_2\leq 1$ for all $i\in[m]$. Let $g:\reals^m\rightarrow\reals$ be convex with bounded subgradients over $[-1,1]^m$. Consider the optimization problem
\begin{align}\label{eq:optprob}
\min_{\x\in\reals^n:\Vert{\x}\Vert_2=1}\{f(\x) := g(\mA(\x))\}.
\end{align}
We assume $g(\cdot)$ is given by a first-order oracle, i.e., for any $\y\in[-1,1]^m$ we can compute some $\g\in\partial{}g(\y)$, where $\partial{}g(\x)$ denotes the subdifferential set of $g$ at $\y$. We assume $\mA$ is also given by a first order oracle that given some $\z\in\reals^n$, returns the vectors $\A_1\z,\dots,\A_m\z$ (it is a first-order oracle in the sense that it is equivalent to taking the gradient of each $h_i(\x) := \frac{1}{2}\x^{\top}\A_i\x$, $i=1,\dots,m$).

Clearly, approximating the leading eigenvalue of a symmetric matrix $\A$ could be written in this form by taking $\mA(\x) = \x^{\top}\A\x$ and $g(x) = -x$. Similarily, the problem of finding a unit vector which satisfies a qaudratic equation of the form $\x^{\top}\A\x = b$ could be written using the same mapping $\mA(\x) = \x^{\top}\A\x$ and $g(x) = (x-b)^2$. Another simple example is that of finding a unit vector $\x$ satisfying $\x^{\top}\A\x\in[a,b]$ for some given scalars $b\geq a$. In this case we take the same operator $\mA$ as before and the function $g(x) = \max\{0,a-x\} + \max\{0,x-b\}$. A more complex example could be of recovering a unit vector $\x$ which (approximately) satisfies a system of equations w.r.t. the magnitudes of its projections onto some unknown orthonormal basis, i.e., find unit vector $\x$ which (approximately) satisfies
\begin{align*}
\sum_{j=1}^n\lambda_i(j)(\v_j^{\top}\x)^2 = \b(i), \quad i=1,\dots,m,
\end{align*}
where $\lambda_i(j), \v_i, i\in[m], j\in[n]$ are not explicitly given, but instead are coded via matrices $\A_1,\dots,\A_m$ such that $\A_i = \sum_{j=1}^n\lambda_i(j)\v_i\v_i^{\top}$, and  for any $\x$, the LHS of equation $i$ above could be evaluated by computing $\x^{\top}\A_i\x$. This leads (for instance) to the optimization problem
\begin{align*}
\min_{\x\in\reals^n:\Vert{\x}\Vert_2=1}\sum_{i=1}^m\vert{\x^{\top}\A_i\x-\b(i)}\vert,
\end{align*}
which is also captured by the model \eqref{eq:optprob}.

As a final example, consider the problem of finding a degree-$k$ polynomial  of some $\A\in\mbS^n$ with coefficients in some convex and compact set $\mC\subset\reals^{k+1}$ with maximal largest eigenvalue. Towards this end, consider the optimization problem:
\begin{align}\label{eq:polyprob}
\max_{\x\in\reals^n:\Vert{\x}\Vert_2=1}\max_{\c\in\mC}\x^{\top}\left({\c(1)\I + \sum_{i=1}^k\c(i+1)\A^i}\right)\x.
\end{align}

This problem could be written in the form \eqref{eq:optprob} by taking $m=k+1$ and $\A_i = \A^{i-1}$ for all $i\in[m]$, and $g(\y) = -\min_{\c\in\mC}(-\c^{\top}\y)$, which is convex in $\y$. If $\x^*$ is an optimal solution to \eqref{eq:polyprob}, then any $\c^*\in\arg\max_{\c\in\mC}\c^{\top}\mA(\x^*)$ corresponds to the coefficients of a polynomial with maximal largest eigenvalue.
% /It follows that for any $\c\in\mC$,
%\begin{align*}
%\lambda_1\left({\c^*(1)\I + \sum_{i=1}^k\c^*(i+1)\A^i}\right) &\geq \x^{*\top}\left({\c^*(1)\I + \sum_{i=1}^k\c^*(i+1)\A^i}\right)\x^*\\
%&\geq \x^{*\top}\left({\c(1)\I + \sum_{i=1}^k\c(i+1)\A^i}\right)\x^*
%\end{align*}

Central to our analysis of Problem \eqref{eq:optprob} would be the following standard semidefinite relaxation:
\begin{align}\label{eq:sdprelax}
\min_{\X\in\mbS^n: \X\succeq 0, \trace(\X)=1}\{F(\X):=g(\sA(\X))\},
\end{align}
where $\sA:\reals^{n\times n}\rightarrow\reals^m$ is the linear mapping $(\sA(\X))(i) = \trace(\X\A_i)$, $i=1,\dots,m$.

As we shall see, the relaxation \eqref{eq:sdprelax} is tight for Problem \eqref{eq:optprob}, in the sense that under our assumption that $\A_1,\dots,\A_m$ have common eigenvectors, \eqref{eq:sdprelax} always admits a rank-one optimal solution (though this solution need not be one of these eigenvectors).

Our Oja-inspired algorithm for Problem  \eqref{eq:optprob} is given below as Algorithm \ref{alg:ojaconv}. It is simply an application of Oja's algorithm (Algorithm \ref{alg:oja}) with a sequence of matrices that corresponds to (minus) subgradients of $F$ (as defined in \eqref{eq:sdprelax}), evaluated at the rank-one matrices corresponding to the iterates of algorithm. Hence, the algorithm is a certain coupling of Oja's algorithm and the subgradient descent method for the convex relaxation \eqref{eq:sdprelax}. Note that Algorithm \ref{alg:ojaconv} (just like Algorithm \ref{alg:oja}) requires to maintain a single vector in memory, and does not require direct access to the underlying matrices $\A_1,\dots,\A_m$, but only via a matrix-vector product (one per matrix per iteration). 

\begin{algorithm}[H]
	\caption{Oja's algorithm meets subgradient descent}\label{alg:ojaconv}
	\begin{algorithmic}
		\STATE \textbf{Input:} step-size $\mu\ge0$, Lipchitz parameter $G\geq 0$ 
		\STATE \textbf{Init:} $\w_1 \gets$ rand unit vector in $\reals^n$
		\FOR{$t = 1,2,...$}
		\STATE $\z_{t}\gets \w_t / \Vert{\w_t}\Vert_2$ 
		\STATE compute $\g_t\in\partial{}g(\mA(\z_t))$ and set $\G_t = -\frac{1}{G}\sum_{i=1}^m\g_t(i)\A_i$
                  \STATE $\w_{t+1} \gets (\I+\mu\G_t)\w_t$   			
        \ENDFOR
	\end{algorithmic}
\end{algorithm}

\begin{theorem}
Consider Problem \eqref{eq:optprob} in case $\A_1,\dots,\A_m\in\mbS^n$ share a full set of orthonormal eigenvectors and $\max_{i\in[m]}\Vert{\A_i}\Vert_2 \leq 1$. Fix $\delta\in(0,1)$. Suppose Algorithm \ref{alg:ojaconv} is run for $T$ iterations with parameter $G$ such that $G \geq \sup_{\x\in[-1,1]^m}\sup_{\g\in\partial{}g(\x)}\Vert{\sum_{i=1}^m\g(i)\A_i}\Vert_2$, with step-size $\mu = \sqrt{\frac{1}{3T}\ln\left({9n^2/\delta}\right)}$, and assume $T$ is sufficiently large so that $\mu \leq 1/2$. Then, with probability at least $1-\delta$ it holds that
\begin{align*}
\min_{t\in[T]}f(\z_t) - \min_{\x\in\reals^n:\Vert{\x}\Vert_2=1}f(\x) \leq G\sqrt{\frac{3\ln\left({9n^2\delta}\right)}{T}}.
\end{align*}
\end{theorem}
\begin{proof}
The proof follows from coupling Theorem \ref{thm:oja2} with a standard online-to-batch argument w.r.t. to the convex relaxation \eqref{eq:sdprelax}.
First, note that under our assumption on the spectral norms of $\A_1,\dots,\A_m$ and the Lipchitz parameter $G$, it follows that for any $t\in[T]$, $\Vert{\G_t}\Vert_2 \leq 1$. Let us apply Theorem \ref{thm:oja2} $n$ times w.r.t. the sequence of matrices $\G_1,\dots,\G_T$, each time w.r.t. one of the  eigenvectors $\v_i$, $i=1,\dots,n$ common to $\A_1,\dots,\A_m$. This yields after slight rearrangement,
\begin{align*}
\forall i\in[n]:\quad \frac{1}{G}\sum_{t=1}^T\trace\left({\left({\z_t\z_t^{\top}-\v_i\v_i^{\top}}\right)\left({-G\cdot \G_t}\right)}\right) \leq \frac{3\mu}{2}{}T - \frac{\ln(\z_1^{\top}\v_i)}{2\mu}.
\end{align*}

Since each $\G_t$ is a linear combination of $\A_1,...,\A_m$, it follows that $\G_1,\dots,\G_T$ also admit the common eigenvectors $\v_1,\dots,\v_n$. This in particular implies that
\begin{align*}
\min_{\X\in\mbS^n:~\X\succeq 0, \trace(\X)=1}\trace\left({\X\left({ - \sum_{t=1}^T\G_t}\right)}\right) = \min_{i\in[n]}\trace\left({\v_i\v_i^{\top}\left({ - \sum_{t=1}^T\G_t}\right)}\right).
\end{align*}

Thus, denoting by $\X^*$ an optimal solution to the convex relaxation \eqref{eq:sdprelax}, we have that

\begin{align*}
\frac{1}{G}\sum_{t=1}^T\trace\left({\left({\z_t\z_t^{\top}-\X^*}\right)\left({ -G\cdot \G_t}\right)}\right) \leq \frac{3\mu}{2}{}T - \min_{i\in[n]}\frac{\ln((\z_1^{\top}\v_i)^2)}{2\mu}.
\end{align*}
Note that for all $t\in[T]$, $-G\cdot\G_t\in\partial{}F(\z_t\z_t^{\top})$, where $F$ is as defined in \eqref{eq:sdprelax} . Thus, using the convexity of $F$ we have that

\begin{align*}
\frac{1}{T}\sum_{t=1}^TF(\z_t\z_t^{\top}) - F(\X^*) \leq \frac{3\mu{}G}{2} - G\min_{i\in[n]}\frac{\ln((\z_1^{\top}\v_i)^2)}{2\mu{}T}.
\end{align*}

Using Lemma \ref{lem:corr} and the union bound w.r.t. $\v_1,\dots,\v_n$, we have that with probability at least $1-\delta$, $\min_{i\in[n]}(\z_1^{\top}\v_i)^2 \geq \frac{\delta}{9n^2}$. Thus, setting $\mu = \sqrt{\frac{1}{3T}\ln\left({9n^2/\delta}\right)}$, we have that with probability at least $1-\delta$,

\begin{align}\label{eq:convres}
\min_{t\in[T]}f(\z_t) - \min_{\x\in\reals^n:\Vert{\x}\Vert_2=1}f(\x) &\leq \min_{t\in[T]}F(\z_t\z_t^{\top}) - F(\X^*) \leq G\sqrt{\frac{3\ln\left({9n^2/\delta}\right)}{T}}.
\end{align}
\end{proof}

\begin{remark}
The second inequality in \eqref{eq:convres} shows in fact that the convex relaxation \eqref{eq:sdprelax} always admits an optimal rank-one solution. Note that this solution however need not be one the eigenvectors in common to $\A_1,\dots,\A_m$. Consider for instance the case $n=m=2$ with $\A_1 = \left( \begin{array}{cc}
1 & 0   \\
0 & 0   \end{array} \right)$, $\A_2 = \left( \begin{array}{cc}
0 & 0   \\
0 & 1   \end{array} \right)$, and $g(\y) = \max_{i\in\{1,2\}}\y(i)$. In this case the common eigenvectors are the standard basis vectors $\e_1, \e_2\in\reals^2$, however it is easy to see that the optimal solutions to Problem \eqref{eq:optprob} in this case are $\left( \begin{array}{c}
\pm 1/\sqrt{2}    \\
\pm 1/\sqrt{2}    \end{array} \right)$.
\end{remark}
\begin{remark}
If the common eigenvectors $\v_1,\dots,\v_n$ of $\A_1,\dots,\A_m$ were known, then by diagonalizing the matrices $\A_1,\dots,\A_m$, Problem \eqref{eq:optprob} could in principle be transformed into the equivalent nonconvex optimization problem $\min_{\y\in\reals^n:\Vert{\y}\Vert_2=1}g(\mB(\y)$, where $\mB:\reals^n\rightarrow\reals^m$ satisfies $(\mB(\y))(i) = \sum_{j=1}^n\lambda_i(j)\y(j)^2$, $i=1,\dots,m$, where for all $i\in[m], j\in[n]$, $\lambda_i(j)$ is the eigenvalue of $\A_i$ corresponding to the eigenvector $\v_j$. Moreover, since for any unit vector $\y\in\reals^n$, the vector $(\y(1)^2,\dots,\y(n)^2)$ is in the unit simplex in $\reals^n$, which we denote by $\Delta_n$, Problem \eqref{eq:optprob} could be further transformed into the convex simplex-constrained optimization problem $\min_{\w\in\Delta_n}g(\Lambda\w)$, where $\Lambda\in\reals^{m\times n}$ is such that $\Lambda(i,j)=\lambda_i(j)$ for all $i\in[m], j\in[n]$. Thus, in case the eigen-basis $\v_1,\dots,\v_n$ is known, Problem  \eqref{eq:optprob} could be solved via the projected subgradient descent method \cite{beck2017first} over the simplex. Interestingly, not only does Algorithm \ref{alg:ojaconv} does not require to compute or store the underlying eigen-basis, but also in terms of the computational complexity of the projection step it is more efficient, requiring to project onto the unit sphere (in $O(n)$ time) rather than on the simplex which requires $O(n\log{n})$ time.
\end{remark}
\subsection{Online learning of eigenvectors with linear memory and runtime?}

In the online learning of eigenvectors problem or online PCA (with a single principal component) \cite{pmlr-v37-garberb15, nie2016online, allen2017follow, pmlr-v99-carmon19a, garber2019regret} which is a sequential prediction task within the framework of online learning, a decision maker (DM) is required to iteratively choose some unit vector $\z_t\in\reals^n$. After choosing her vector, a positive-semidefinite matrix $\A_t\in\mbS^n_{+}$ is revealed and the DM obtains payoff of $\z_t^{\top}\A_t\z_t$. The goal is to guarantee sublinear (in $T$) regret, that is, to have 
\begin{align*}
\lambda_1\left({\sum_{t=1}^T\A_t}\right) - \sum_{t=1}^T\z_t^{\top}\A_t\z_t = o(T).
\end{align*}
Since the payoff functions are not concave but convex in the decision variable and the feasible set is also non-convex, in its natural formulation this problem does not fall directly into the paradigm of online convex optimization which allows for efficient online algorithms \cite{shalev2012online, hazan2016introduction}. Nevertheless, by considering a natural semidefinite relaxation, in which the decision variable $\z_t$ is replaced with a positive semidefinite unit-trace matrix $\Z_t$ (which yields a convex feasible set), and accordingly the payoff on each round $t$ becomes the linear (in $\Z_t$) function $\trace(\Z_t\A_t)$, one can obtain a convex formulation \cite{pmlr-v37-garberb15, nie2016online, allen2017follow, pmlr-v99-carmon19a, garber2019regret}. While significant progress on efficient algorithms for this problem has been made in recent years, in particular with the works \cite{allen2017follow, pmlr-v99-carmon19a}, all algorithms which guarantee sublinear regret require to store $n\times n$ matrices in memory and a leading eigenvector computation or a similarly expensive exponentiated matrix-vector product on each round  \cite{allen2017follow, pmlr-v99-carmon19a}.

A natural question, which was already proposed in \cite{pmlr-v37-garberb15, garber2019regret}, is whether Oja's algorithm (Algorithm \ref{alg:oja}), which only maintains a vector in $\reals^n$ in memory and requires only the first-order information $\A_t\z_t$ on each round $t$ (without even having full access to  $\A_t$) could be proved to guarantee  sublinear regret. 
Note that if the payoff matrices are chosen \textit{adaptively}, i.e., $\A_t$ may depend on previous actions, then the answer is trivially negative since upon observing the first vector $\z_1$, one can choose $\A_1,\dots,\A_T$ to be orthogonal to $\z_1$, in which case Oja's algorithm trivially fails. However, the question remain open in the \textit{oblivious} setting, in which $\A_1,\dots,\A_T$ are chosen beforehand and without knowledge of the initial (random) vector $\z_1$. 

\begin{conjecture}
For any $T$ large enough and any fixed step-size $\mu$, there exists a sequence of matrices $\A_1,\dots\A_T$ in $\mbS^n$ with spectral norm at most $1$, such that Oja's algorithm (Algorithm \ref{alg:oja})  incurs linear (in  $T$) regret  with at least constant probability.
\end{conjecture}

Since Oja's algorithm only maintains a one-dimensional subspace at each time, it might seem that a simple construction based on a fixed orthonormal basis, i.e., setting $\A_t = \v_i\v_i^{\top}$ for some $i\in[n]$ and a fixed orthonormal set $\v_1,\dots,\v_n$, will be sufficient to prove the conjecture. However, similarly to the analysis in Section \ref{sec:EVapprox}, Theorems \ref{thm:oja} and \ref{thm:oja2} (together with Lemma \ref{lem:corr}) imply that with a suitable choice of step-size, even in a more general case that $\A_1,\dots,\A_T$ have common eigenvectors (but not necessarily rank-one), Oja's algorithm indeed guarantees sublinear regret (scaling with $\sqrt{T}$) with high-probability. Thus, proving the conjecture, if true, requires more sophisticated constructions.

%\subsubsection{Online null-space learning in the realizable case} 

%Consider now a related problem in which given a sequence of positive semidefinite matrices $\A_1,\dots,\A_T\in\mbS^n_+$ our goal is to iteratively choose a unit vector $\z_t, t\in[T]$ that is competitive with an eigenvector corresponding to the smallest eigenvalue of the cumulative matrix $\sum_{t=1}^T\A_t$. For instance, we can think of $\z^{\top}\A_t\z$ has measuring the variance of some unit vector $\z$ in round $t$, and our goal is to compete with a direction in space of overall minimal variance.
%Ideally,  we shall have $\lambda_{\min}(\sum_{t=1}^T\A_t) = 0$, which we refer to as the \textit{realizable setting}, in which case our goal is to online learn a vector in the joint null-space $\bigcap_{t=1}^T\textrm{nullspace}(\A_t) \neq \{\mathbf{0}\}$. In this case, denoting by $\v$ some unit vector in $\bigcap_{t=1}^T\textrm{nullspace}(\A_t)$, under our standard assumption that $\Vert{\A_t}\Vert_2\leq 1, t\in[T]$, by applying Theorem \ref{thm:oja2} together with Lemma \ref{lem:corr} w.r.t to $\v$ and tuning $\eta$ appropriately, yields a $O(\sqrt{T})$ regret bound, without the assumption that $\A_1,\dots,\A_T$ admit a common full set of $n$ eigenvectors.

\section{Discussion and Related Work}

Our observation regarding the equivalence of Oja's algorithm to MW (in case of common eigenvectors) could be related to a recent line of results \cite{gunasekar2017implicit, amid2020winnowing, amid2020reparameterizing, ghai2022non} that studied equivalences between (Euclidean) gradient descent and non-Euclidean mirror descent on a  transformed problem or their continuous-time counterparts, however, to the best of our knowledge the spectral setting of optimization with quadratic forms over the unit sphere discussed here, was not previously considered. 

Also relevant to the problems discussed here, is the well known matrix generalization of the MW method, namely the matrix multiplicative weights algorithm (MMW), see for instance \cite{tsuda2005matrix, warmuth2006online, arora2007combinatorial, JMLR:v9:warmuth08a, allen2017follow, pmlr-v99-carmon19a}. MMW is an online algorithm for optimization over the the set of unit-trace real positive semidefinite matrices, which naturally underlies convex relaxations for problems such as considered here, e.g, the convex relaxation \eqref{eq:sdprelax}. This algorithm guarantees regret bounds that hold for any sequence of bounded symmetric matrices, in particular without further assumption on  common eigenvectors. However, it is significantly more complex, both in terms of computational complexity, requiring to store $n\times n$ matrices in memory and computing either an eigen-decomposition or a product of an exponentiated matrix with a vector on each iteration (both of which are more complex than the fundamental problem of leading eigenvalue computation which is one of our main motivations here) \cite{allen2017follow, pmlr-v99-carmon19a}, and it relies on a  more involved regret analysis than the elementary proof of its simplex counterpart MW.  It is worth noting that recently \cite{allen2017follow} used a MMW-inspired approach to prove the convergence of Oja's algorithm in the stochastic setting (discussed at the introduction) with an elegant and relatively short proof, but not as elementary as the arguments brought here. Extending our MW-based approach to this stochastic setting (without the simplifying assumption that the individual samples share common eigenvectors) seems difficult and given the already short and accessible proof of \cite{allen2017follow} we do not pursue this direction.

Finally, given some $\A\in\mbS^n$, one is often not only interested in computing the leading eigenvector, but the entire leading subspace of some dimension $k \leq n$, i.e., to find vectors $\z_1,\dots,\z_k\in\reals^n$ such that $\textrm{span}\{\z_1,\dots,\z_k\} = \textrm{span}\{\v_1,\dots,\v_k\}$, where $\v_1,\dots,\v_k$ are eigenvectors of $\A$ corresponding to the top $k$ eigenvalues. While the power method admits a well known generalization for extracting the top subspace which is based on the incorporation of a QR-decomposition step on each iteration (see for instance  \cite{golub2013matrix, saad2011numerical}), it seems difficult to extend the arguments brought here (for instance using the  ideas in \cite{JMLR:v9:warmuth08a} which generalized the problem of prediction with expert advice to prediction with a subset of experts) to this setting and it is left as a possible direction for future work.

\bibliography{bibs}
\bibliographystyle{plain}

\end{document}